\documentclass[a4paper , 11pt]{amsart}
\usepackage[top=30truemm,bottom=35truemm,left=30truemm,right=30truemm]{geometry}
\usepackage{amsmath}
\usepackage{amsthm}
\usepackage{amsfonts}
\usepackage{bm}

\newtheorem{thm}{Theorem}[section]
\newtheorem{lem}[thm]{Lemma}
\newtheorem{prop}[thm]{Proposition}
\newtheorem{cor}[thm]{Corollary}
\theoremstyle{definition}
\newtheorem{prob}[thm]{Problem}
\newtheorem*{rem}{Remark}

\DeclareMathOperator{\insep}{insep}

\newcommand{\setmid}{\mathrel{}\middle|\mathrel{}}%集合の区切り
\newcommand{\abs}[1]{\left\lvert#1\right\rvert}%絶対値

\makeatletter %\fB => \mathfrak{B}, \gb => \mathfrak{b}, \lB => \mathbb{B}, \cB => \mathcal{B}
\@tempcnta\z@
\loop\ifnum\@tempcnta<26
\advance\@tempcnta\@ne
\expandafter\edef\csname f\@Alph\@tempcnta\endcsname{\noexpand\mathfrak{\@Alph\@tempcnta}}
\expandafter\edef\csname l\@Alph\@tempcnta\endcsname{\noexpand\mathbb{\@Alph\@tempcnta}}
\expandafter\edef\csname c\@Alph\@tempcnta\endcsname{\noexpand\mathcal{\@Alph\@tempcnta}}
\repeat

\title{On Diophantine exponents for Laurent series over a finite field}
\author{Tomohiro Ooto}
\date{}

\begin{document}
	
\address{Graduate School of Pure and Applied Sciences, University of Tsukuba, Tennodai 1-1-1, Tsukuba, Ibaraki, 305-8571, Japan}
\email{ooto@math.tsukuba.ac.jp}

\begin{abstract}
In this paper, we study the properties of Diophantine exponents $w_n$ and $w_n^{*}$ for Laurent series over a finite field.
We prove that for an integer $n\geq 1$ and a rational number $w>2n-1$, there exist a strictly increasing sequence of  positive integers $(k_j)_{j\geq 1}$ and a sequence of algebraic Laurent series $(\xi _j)_{j\geq 1}$ such that $\deg \xi _j =p^{k_j}+1$ and
\begin{equation*}
w_1(\xi _j)=w_1 ^{*}(\xi _j)=\ldots =w_n(\xi _j)=w_n ^{*}(\xi _j)=w
\end{equation*}
for any $j\geq 1$.
For each $n\geq 2$, we give explicit examples of Laurent series $\xi $ for which $w_n(\xi )$ and $w_n^{*}(\xi )$ are different.
\end{abstract}

\maketitle

\section{Introduction}\label{sec:intro}

Mahler \cite{Mahler32} and Koksma \cite{Koksma39} introduced Diophantine exponents which measure the quality of approximation to real numbers.
Using the Diophantine exponents, they classified the set $\lR$ all of real numbers.
Let $\xi $ be a real number and $n\geq 1$ be an integer.
We denote by $w_n(\xi )$ the supremum of the real numbers $w$ which satisfy
\begin{equation*}
0<|P(\xi )|\leq H(P)^{-w}
\end{equation*}
for infinitely many integer polynomials $P(X)$ of degree at most $n$.
Here, $H(P)$ is defined to be the maximum of the absolute values of the coefficients of $P(X)$.
We denote by $w_n ^{*}(\xi )$ the supremum of the real numbers $w^{*}$ which satisfy
\begin{equation*}
0<|\xi -\alpha |\leq H(\alpha )^{-w^{*}-1}
\end{equation*}
for infinitely many algebraic numbers $\alpha $ of degree at most $n$.
Here, $H(\alpha )$ is equal to $H(P)$, where $P(X)$ is the minimal polynomial of $\alpha $ over $\lZ$.

We recall some results on Diophantine exponents.
It is clear that $w_1(\xi )=w_1 ^{*}(\xi )$ for all real numbers $\xi $.
Roth \cite{Roth55} established that $w_1(\xi )=w_1^{*}(\xi )=1$ for all irrational algebraic real numbers $\xi $.
Furthermore, it follows from the Schmidt Subspace Theorem that
\begin{equation}\label{eq:SST}
w_n(\xi )=w_n ^{*}(\xi )=\min \{ n,d-1\} 
\end{equation}
for all $n\geq 1$ and algebraic real numbers $\xi $ of degree $d$.
It is known that
\begin{equation*}
0\leq w_n(\xi )-w_n ^{*}(\xi )\leq n-1
\end{equation*}
for all $n\geq 1$ and real numbers $\xi $ (see Section 3.4 in \cite{Bugeaud04}).
Sprind\u{z}uk \cite{Sprindzuk69} proved that $w_n(\xi )=w_n ^{*}(\xi )=n$ for all $n\geq 1$ and almost all real numbers $\xi $.
Baker \cite{Baker76} proved that for $n\geq 2$, there exists a real number $\xi $ for which $w_n(\xi )$ and $w_n^{*}(\xi )$ are different.
More precisely, he proved that the set of all values taken by the function $w_n-w_n^{*}$ contains the set $[0,(n-1)/n]$ for $n\geq 2$.
In recent years, this result has been improved.
Bugeaud \cite{Bugeaud12, Bugeaud042} showed that the set of all values taken by $w_2-w_2^{*}$ is equal to the closed interval $[0,1]$ and the set of all values taken by $w_3-w_3^{*}$ contains the set $[0,2)$.
Bugeaud and Dujella \cite{Bugeaud11} proved that for any $n\geq 4$, the set of all values taken by $w_n-w_n^{*}$ contains the set $\left[ 0,\frac{n}{2}+\frac{n-2}{4(n-1)}\right) $.

Let $p$ be a prime.
We can define Diophantine exponents $w_n$ and $w_n^{*}$ over the field $\lQ _p$ of $p$-adic numbers in a similar way to the real case.
Analogues of the above results for $p$-adic numbers have been studied (see e.g.\ Section 9.3 in \cite{Bugeaud04} and \cite{Bugeaud15, Pejkovic12}).

Let $p$ be a prime and $q$ be a power of $p$.
Let us denote by $\lF _q$ the finite field of $q$ elements, $\lF _q[T]$ the ring of all polynomials over $\lF _q$, $\lF _q(T)$ the field of all rational functions over $\lF _q$, and $\lF _q((T^{-1}))$ the field of all Laurent series over $\lF _q$.
For $\xi \in \lF _q((T^{-1})) \setminus \{ 0\}$, we can write
\begin{equation*}
\xi = \sum _{n=N}^{\infty }a_n T^{-n},
\end{equation*}
where $N\in \lZ $, $a_n \in \lF _q$ for all $n\geq N$, and $a_N \neq 0$.
We define an absolute value on $\lF _q ((T^{-1}))$ by $|0|:=0$ and $|\xi |:=q^{-N}$.
This absolute value can be uniquely extended to the algebraic closure of $\lF _q((T^{-1}))$ and we continue to write $|\cdot |$ for the extended absolute value.
We call an element of $\lF _q((T^{-1}))$ an {\itshape algebraic Laurent series} if the element is algebraic over $\lF _q(T)$.
We can define Diophantine exponents $w_n$ and $w_n^{*}$ for Laurent series over a finite field in a similar way to the real case.

Mahler \cite{Mahler49} proved that an analogue of the Roth Theorem in this framework does not hold, that is, there exists an algebraic Laurent series $\xi $ such that $w_1(\xi )>1$.
Indeed, let $r$ be a power of $p$ and put $\xi :=\sum _{n=0}^{\infty }T^{-r^n}$.
Then $\xi $ is an algebraic Laurent series of degree $r$ with $w_1(\xi )=r-1$.
After that, several people investigated algebraic Laurent series for which the analogue of the Roth Theorem does not hold (see e.g.\ \cite{Chen13, Firicel13, Mathan92, Schmidt00, Thakur99}).
Furthermore, Thakur \cite{Thakur11, Thakur13} constructed explicit algebraic Laurent series for which the analogue of \eqref{eq:SST} for $n=r+1$ does not hold, where $r$ is a power of $p$.
For example, for integers $m,n\geq 2$, he constructed algebraic Laurent series $\alpha _{m,n}$, with explicit equations and continued fractions, such that
\begin{equation*}
	\deg \alpha _{m,n}\leq r^m+1,\quad \lim_{n\rightarrow \infty }E_1(\alpha _{m,n})=2,\quad \lim_{n\rightarrow \infty }E_{r+1}(\alpha _{m,n})\geq r^{m-1}+\frac{r-1}{(r+1)r}.
\end{equation*}
Here, $E_n(\xi )$ measures the quality of approximations to $\xi $ by algebraic Laurent series of degree $n$ (see Section \ref{sec:mainresult} for the precise definition and relation between $E_n$ and $w_n^{*}$).
Since $w_{r+1}^{*}(\alpha _{m,n})+1\geq E_{r+1}(\alpha_{m,n})$, we obtain
\begin{equation*}
	\lim_{n\rightarrow \infty }w_{r+1}^{*}(\alpha _{m,n})\geq r^{m-1}-\frac{r^2+1}{(r+1)r},
\end{equation*}
which implies that an analogue of \eqref{eq:SST} does not hold for $\alpha _{m,n}$ with sufficiently large $m,n$.
In this paper, we investigate the phenomenon that properties of the Diophantine exponents in characteristic zero are different from that of positive characteristic.
More precisely, for an integer $n\geq 1$, we construct algebraic Laurent series $\xi $ such that $E_1(\xi )$ are large values and
\begin{equation*}
	E_1(\xi ) = \max\{ E_m(\xi )\mid 1\leq m\leq n \} .
\end{equation*}

The author \cite{Ooto17} showed that there exists $\xi \in \lF _q((T^{-1}))$ for which $w_2(\xi )$ and $w_2^{*}(\xi )$ are different.
In this paper, improving the method in \cite{Ooto17}, we solve the problem of whether or not there exists $\xi \in \lF _q((T^{-1}))$ for which $w_n(\xi )$ and $w_n^{*}(\xi )$ are different for any $n\geq 3$.

This paper is organized as follows:
We state the main results on the Diophantine exponents $w_n$ and $w_n ^{*}$ for Laurent series over a finite field in Section \ref{sec:mainresult}.
In Section \ref{sec:pre}, we prepare some lemmas in order to prove the main results.
We collect the proofs of the main results in Section \ref{sec:mainproof}.
It is well-known that finite automatons relate to algebraic Laurent series.
In Section \ref{sec:conrem}, we give properties of the Diophantine exponents for Laurent series whose coefficients are generated by finite automatons.
We also give analogues of the main results for real and $p$-adic numbers.

\section{Main results}\label{sec:mainresult}

In this section, we state the main results about the Diophantine exponents for Laurent series over a finite field and give some problems associated to the main results.

We use the following notation throughout this paper.
We denote by $\lfloor x\rfloor $ the integer part of a real number $x$.
We use the Vinogradov notation $A\ll B$ if $|A|\leq c |B|$ for some constant $c>0$.
We write $A\asymp B$ if $A\ll B$ and $B\ll A$.
For a finite word $W$, we put $\overline{W}:=WW\cdots W\cdots $ (infinitely many concatenations of the word $W$).
An infinite word ${\bf a}=a_0a_1\cdots $ is called {\itshape ultimately periodic} if there exist a finite word $U$ and a non-empty finite word $V$ such that ${\bf a}=U\overline{V}$.
We identify a sequence $(a_n)_{n\geq 0}$ with the infinite word $a_0 a_1 \cdots a_n \cdots $.

We denote by $(\lF _q[T])[X]$ the set of all polynomials in $X$ over $\lF_q[T]$.
The {\itshape height} of a polynomial $P(X)\in (\lF _q [T])[X]$, denoted by $H(P)$, is defined to be the maximum of the absolute values of the coefficients of $P(X)$.
We denote by $(\lF _q[T])[X]_{\min }$ the set of all non-constant, irreducible, primitive polynomials in $(\lF _q[T])[X]$ whose leading coefficients are monic polynomials in $T$.
For $\alpha \in \overline{\lF _q(T)}$, there exists a unique $P(X)\in (\lF _q[T])[X]_{\min }$ such that $P(\alpha )=0$. 
We call the polynomial $P(X)$ the {\itshape minimal polynomial} of $\alpha $.
The {\itshape height} (resp.\ the {\itshape degree}, the {\itshape inseparable degree}) of $\alpha $, denoted by $H(\alpha )$ (resp.\ $\deg \alpha $, $\insep \alpha $), is defined to be the height of $P(X)$ (resp.\ the degree of $P(X)$, the inseparable degree of $P(X)$).
We now define the Diophantine exponents for Laurent series over a finite field.
Let $n\geq 1$ be an integer and $\xi $ be in $\lF _q((T^{-1}))$.
We denote by $w_n(\xi )$ (resp.\ $w_n^{*}(\xi )$) the supremum of the real numbers $w$ (resp.\ $w^{*}$) which satisfy 
\begin{equation*}
0<|P(\xi )|\leq H(P)^{-w}\quad (\text{resp.\ } 0<|\xi -\alpha |\leq H(\alpha )^{-w^{*}-1})
\end{equation*}
for infinitely many $P(X)\in (\lF _q[T])[X]$ of degree at most $n$ (resp.\ $\alpha \in \overline{\lF _q(T)}$ of degree at most $n$).
It is clear that $w_1(\xi )=w_1 ^{*}(\xi )$ for all $\xi \in \lF _q((T^{-1}))$.

Let $n\geq 1$ be an integer and let $\xi \in \lF_q((T^{-1}))$ be a Laurent series which is not algebraic of degree at most $n$.
We denote by $E_n(\xi )$ the supremum of the real numbers $w$ which satisfy 
\begin{equation*}
0<|\xi -\alpha |\leq H(\alpha )^{-w}
\end{equation*}
for infinitely many algebraic Laurent series $\alpha \in \lF _q((T^{-1}))$ of degree $n$.
It is obvious that 
\begin{equation}\label{eq:e1}
	E_1(\xi )=w_1(\xi )+1=w_1 ^{*}(\xi )+1
\end{equation}
for all irrational Laurent series $\xi \in \lF _q((T^{-1}))$.
It is also obvious that
\begin{equation}\label{eq:e12}
w_n^{*}(\xi )+1\geq \max\{E_m(\xi )\mid 1\leq m\leq n\}
\end{equation}
for all $\xi \in \lF _q((T^{-1}))$ which are not algebraic of degree at most $n$.

As in the classical continued fraction theory of real numbers, if $\xi \in \lF _q((T^{-1}))$, then we can write
\begin{equation*}
\xi = a_0+\cfrac{1}{a_1+\cfrac{1}{a_2+\cfrac{1}{\cdots }}},
\end{equation*}
where $a_0, a_n \in \lF _q[T]$ and $\deg a_n\geq 1$ for all $n\geq 1$.
We write $\xi =[a_0,a_1,a_2,\ldots ]$.
We call $a_0 $ and $a_n $ the {\itshape partial quotients} of $\xi $.

Schmidt \cite{Schmidt00} and Thakur \cite{Thakur99} studied the Diophantine exponent $w_1$ for algebraic Laurent series.
Schmidt \cite{Schmidt00} introduced a classification of algebraic Laurent series as follows:
Let $\alpha $ be in $\lF _q((T^{-1}))\setminus \lF _q(T)$.
We say that $\alpha $ is of {\itshape Class I} (resp. {\itshape Class IA}) if there exist an integer $s\geq 0$ and $R,S,T,U \in \lF _q[T]$ with $RU-ST \neq 0$ (resp. $RU-ST \in \lF _q ^{\times }$) such that
\begin{equation*}
\alpha =\frac{R \alpha ^{p^s} +S}{T \alpha ^{p^s}+U}.
\end{equation*}
For example, any quadratic Laurent series is of Class IA.
Mathan \cite{Mathan92} proved that the value of $w_1$ for a Laurent series of Class I is rational.
However, it is not known whether or not there exists an algebraic Laurent series for which the value of $w_1$ is irrational.
Let $r$ be a power of $p$.
Schmidt \cite{Schmidt00} and Thakur \cite{Thakur99} independently proved that for any rational number $1<w\leq r$, there exists an algebraic Laurent series $\xi \in \lF _q((T^{-1}))$ of degree at most $r+1$ such that $w_1(\xi )=w$.
It is well-known that $w_1(\xi )=1$ and $w_1(\eta )\geq 1$ hold for any quadratic Laurent series $\xi \in \lF _q((T^{-1}))$ and irrational Laurent series $\eta \in \lF _q((T^{-1}))$ (see e.g.\ Theorem 5.1 in \cite{Ooto17} and Lemma \ref{lem:alg}).
Therefore, we deduce that the set of all values taken by $w_1$ over the set of all Laurent series of Class IA is equal to the set of all rational numbers greater than or equal to $1$.
Chen \cite{Chen13} refined Schmidt and Thakur's result by showing that the degree of $\xi $ can be taken to be equal to $r+1$.

We partially extend Chen's result to $w_n$ and $w_n ^{*}$ for $n\geq 2$.

\begin{thm}\label{thm:mainalg}
Let $d\geq 1$ be an integer and $w>2d-1$ be a rational number.
Then there exist a strictly increasing sequence of  positive integers $(k_j)_{j\geq 1}$ and a sequence $(\xi _j)_{j\geq 1}$ such that, for any $j\geq 1$, $\xi _j$ is of Class IA of degree $p^{k_j}+1$, and
\begin{equation}\label{eq:mainlabel}
w_1(\xi _j)=w_1 ^{*}(\xi _j)=\ldots =w_d(\xi _j)=w_d ^{*}(\xi _j)=w.
\end{equation}
\end{thm}

\begin{rem}
(i).\ By Lemma \ref{lem:alg}, we obtain $w\leq p^{k_1}$.

(ii).\ When $w=p^s$, where $s\geq 1$ is an integer, it is known that there exist Laurent series of Class IA which satisfy \eqref{eq:mainlabel}.
Indeed, let $a\in \lF _q[T]$ be a non-constant polynomial.
We define a Laurent series $\xi _a$ of Class IA by
\begin{equation*}
\xi _a=[a,a^{p^s},a^{p^{2s}},\ldots ].
\end{equation*}
Then it is known that $\xi _a$ is algebraic of degree $p^s+1$ and satisfies $w_1(\xi _a)=p^s$ (see Theorem 1 (1) and Remarks (1) in \cite{Thakur99}).
Therefore, it follows from Lemma \ref{lem:alg} that
\begin{equation*}
w_n(\xi _a)=w_n^{*}(\xi _a)=p^s
\end{equation*}
for all $n\geq 1$.
\end{rem}

(iii).\ By \eqref{eq:e1} and \eqref{eq:e12}, we deduce that
\begin{equation*}
	E_1(\xi_j) = \max\{E_n(\xi_j)\mid 1\leq n\leq d\}=w+1,
\end{equation*}
where the $\xi _j$'s are algebraic Laurent series as in Theorem \ref{thm:mainalg}.

It is known that the values of $w_1$ can be determined through the partial quotients of continued fraction.
In this paper, we extend this result to $w_n$ and $w_n ^{*}$ for all $n\geq 1$ for a certain class of Laurent series.
This is the key point of the proof of Theorem \ref{thm:mainalg}.

As mentioned in Section \ref{sec:intro}, it is known that $w_n(\xi )=w_n^{*}(\xi )$ for all real algebraic numbers $\xi $ and integers $n\geq 1$.
The proof of this result depends on the Schmidt Subspace Theorem which is a generalization of the Roth Theorem.
However, analogues of these theorems in positive characteristic do not hold (see Section \ref{sec:intro}).
Therefore, we address the following problem.

\begin{prob}\label{prob:normalstar}
Is it true that
\begin{equation*}
w_n(\xi )=w_n ^{*}(\xi )
\end{equation*}
for an integer $n\geq 1$ and an algebraic Laurent series $\xi $?
\end{prob}

Note that Theorem \ref{thm:mainalg} gives a partial answer to Problem \ref{prob:normalstar}.
If we remove the condition that $\xi $ is algebraic, then the answer to Problem \ref{prob:normalstar} is not true (see Theorems \ref{thm:mainconti1} and \ref{thm:mainconti2} below).

We state some corollaries of Theorem \ref{thm:mainalg}.

\begin{cor}\label{cor:classcor}
Let $n\geq 1$ be an integer.
Then the set of all values taken by $w_n$ {\normalfont (}resp. $w_n ^{*}${\normalfont )} over the set of all Laurent series of Class IA contains the set of all rational numbers greater than $2n-1$.
\end{cor}

We address the following natural problem arising from Corollary \ref{cor:classcor}.

\begin{prob}
Let $n\geq 1$ be an integer.
Determine the set of all values taken by $w_n$ (resp. $w_n ^{*}$) over the set of all algebraic Laurent series.
\end{prob}

Since the sequence of degrees of $\xi _j$ tends to infinity under the conditions of Theorem \ref{thm:mainalg}, we obtain the following corollary.

\begin{cor}\label{cor:linind}
Let $d\geq 1$ be an integer and $w>2d-1$ be a rational number.
Then there exists a set $\{ \xi _j \mid j\geq 1 \}$ of linearly independent Laurent series of Class IA such that, for any $j\geq 1$
\begin{equation*}
w_1(\xi _j)=w_1 ^{*}(\xi _j)=\ldots =w_d(\xi _j)=w_d ^{*}(\xi _j)=w.
\end{equation*}
\end{cor}

We obtain the following theorem in a similar method to the proof of Theorem \ref{thm:mainalg}.

\begin{thm}\label{thm:realequal}
Let $d\geq 1$ be an integer and $w\geq 2d-1$ be a real number.
Then there exist uncountably many $\xi \in \lF _q((T^{-1}))$ such that
\begin{equation*}
w_1(\xi )=w_1 ^{*}(\xi )=\ldots =w_d(\xi )=w_d ^{*}(\xi )=w.
\end{equation*}
\end{thm}

Analogues of Theorem \ref{thm:realequal} for real and $p$-adic numbers are already given in \cite{Bugeaud10, Bugeaud112}.

For $\xi \in \lF _q((T^{-1}))$, it is easily seen that $(w_n(\xi ))_{n\geq 1}$ and $(w_n^{*}(\xi ))_{n\geq 1}$ are increasing sequences with $0\leq w_n(\xi ), w_n^{*}(\xi )\leq +\infty $ for all $n\geq 1$.
Therefore, we have $w_n(\xi )=w_n ^{*}(\xi )=0$ and $w_n(\eta )=w_n^{*}(\eta )=1$ for all $n\geq 1, \xi \in \lF _q(T)$, and quadratic Laurent series $\eta \in \lF _q((T^{-1}))$ by Theorem 5.1 in \cite{Ooto17} and Lemma \ref{lem:alg}.
It is immediate that for any $n\geq 1, w_n(\xi )=w_n^{*}(\xi )=+\infty $, where $\xi =\sum _{m=1}^{\infty }T^{-m!}$.
Hence, we have the following corollary of Theorem \ref{thm:realequal}.

\begin{cor}\label{cor:dis}
For an integer $n\geq 1$, the set of all values taken by $w_n$ {\normalfont (}resp.\ $w_n^{*}${\normalfont )} contains the set $\{ 0,1\} \cup [2n-1,+\infty ]$.
Furthermore, the set of all values taken by $w_1$ {\normalfont (}resp.\ $w_1^{*}${\normalfont )} is equal to $\{ 0\}\cup [1,+\infty ]$.
\end{cor}

We extend Theorems 1.1 and 1.2 in \cite{Ooto17}.

\begin{thm}\label{thm:mainconti1}
Let $d\geq 2$ be an integer and $w\geq (3d+2+\sqrt{9d^2+4d+4})/2$ be a real number.
Let $a,b\in \lF_q[T]$ be distinct non-constant polynomials.
Let $(a_{n,w})_{n\geq 1}$ be the sequence given by
\begin{equation*}
a_{n,w}=\begin{cases}
b & \text{if } n=\lfloor w^i \rfloor \text{ for some integer } i\geq 0,\\
a & \text{otherwise}.
\end{cases}
\end{equation*}
Set $\xi _w:=[0,a_{1,w},a_{2,w},\ldots ]$.
Then we have 
\begin{equation}\label{eq:mainconti1}
w_n ^{*}(\xi _w)=w-1,\quad w_n(\xi _w)=w 
\end{equation}
for all $2\leq n\leq d$.
\end{thm}

\begin{thm}\label{thm:mainconti2}
Let $d\geq 2$ be an integer, $w\geq 121d^2$ be a real number, and $a,b,c\in \lF_q[T]$ be distinct non-constant polynomials.
Let $0<\eta <\sqrt{w}/d$ be a positive number and put $m_i:=\lfloor (\lfloor w^{i+1}\rfloor -\lfloor w^i-1\rfloor )/\lfloor \eta w^i\rfloor \rfloor $ for all $i\geq 1$.
Let $(a_{n,w,\eta })_{n\geq 1}$ be the sequence given by
\begin{equation*}
a_{n,w,\eta }=\begin{cases}
b & \text{if } n=\lfloor w^i \rfloor \text{ for some integer } i\geq 0,\\
c & \parbox{250pt}{$\text{if } n \neq \lfloor w^i\rfloor \text{ for all integers }i\geq 0 \text{ and } n=\lfloor w^j\rfloor +m \lfloor \eta w^j\rfloor \text{ for some integers }1\leq m\leq m_j, j\geq 1,$}\\
a & \text{otherwise}.
\end{cases}
\end{equation*}
Set $\xi _{w,\eta }:=[0,a_{1,w,\eta },a_{2,w,\eta },\ldots ]$.
Then we have
\begin{equation}\label{eq:mainconti2}
w_n ^{*}(\xi _{w,\eta })=\frac{2 w-2-\eta }{2+\eta },\quad w_n(\xi _{w,\eta })=\frac{2 w-\eta }{2+\eta }
\end{equation}
for all $2\leq n\leq d$.
Hence, we have
\begin{equation*}
w_n(\xi _{w,\eta })-w_n ^{*}(\xi _{w,\eta })=\frac{2}{2+\eta } 
\end{equation*}
for all $2\leq n\leq d$.
\end{thm}

Theorems \ref{thm:mainconti1} and \ref{thm:mainconti2} imply that for each $n\geq 2$, we can construct explicitly Laurent series $\xi $ for which $w_n(\xi )$ and $w_n ^{*}(\xi )$ are different.
The general strategies of the proof of Theorems \ref{thm:mainconti1} and \ref{thm:mainconti2} are same as that of Theorems 1.1 and 1.2 in \cite{Ooto17}.
The key ingredient is that for $n\geq 3$, if $\xi \in \lF _q((T^{-1}))$ has a dense (in a suitable sense) sequence of very good quadratic approximations, then we can determine $w_n(\xi )$ and $w_n^{*}(\xi )$.

The following corollary is immediate from Theorems \ref{thm:realequal}, \ref{thm:mainconti1}, and \ref{thm:mainconti2}.

\begin{cor}\label{cor:minusset}
Let $d\geq 2$ be an integer and $\delta $ be in the closed interval $[0,1]$.
Then there exist uncountably many $\xi \in \lF _q((T^{-1}))$ such that $w_n(\xi )-w_n ^{*}(\xi )=\delta $ for all $2\leq n\leq d$.
In particular, the set of all values taken by $w_d-w_d ^{*}$ contains the closed interval $[0,1]$.
\end{cor}

Note that it is already known that the set of all values taken by $w_2-w_2^{*}$ is the closed interval $[0,1]$ in \cite{Ooto17}.

In the last part of this section, we mention a problem associated to Corollary \ref{cor:dis} and \ref{cor:minusset}.

\begin{prob}
Let $n\geq 1$ be an integer.
Determine the set of all values taken by $w_n$ (resp.\ $w_n^{*}, w_n-w_n^{*}$).
\end{prob}

\section{Preliminaries}\label{sec:pre}

Let $\xi $ be in $\lF _q((T^{-1}))$ and $n\geq 1$ be an integer.
We denote by $\tilde{w}_n(\xi )$ the supremum of the real numbers $w$ which satisfy
\begin{equation*}
0<|P(\xi )|\leq H(P)^{-w}
\end{equation*}
for infinitely many $P(X) \in (\lF _q[T])[X]_{\min }$ of degree at most $n$.

\begin{lem}\label{lem:weak}
Let $n\geq 1$ be an integer and $\xi $ be in $\lF _q((T^{-1}))$.
Then we have
\begin{equation*}
w_n (\xi )=\tilde{w}_n (\xi ).
\end{equation*}
\end{lem}

\begin{proof}
See Lemma 5.3 in \cite{Ooto17}.
\end{proof}

\begin{lem}\label{lem:poly}
Let $\alpha ,\beta $ be in $\lF _q ((T^{-1}))$ and $P(X) \in (\lF _q[T])[X]$ be a non-constant polynomial of degree $d$.
Let $C\geq 0$ be a real number.
Assume that $|\alpha -\beta |\leq C$.
Then there exists a positive constant $C_1(\alpha ,C,d)$, depending only on $\alpha ,C,$ and $d$, such that
\begin{equation}\label{eq:poly.differ}
|P(\alpha )-P(\beta )|\leq C_1(\alpha ,C,d)|\alpha -\beta |H(P).
\end{equation}
\end{lem}

\begin{proof}
It is easily seen that
\begin{equation*}
|P(\alpha )-P(\beta )|\leq H(P)\max _{1\leq i\leq d} |\alpha ^i-\beta ^i|.
\end{equation*}
By the assumption, we have $\max (C,|\alpha |)=\max (C,|\beta |)$.
For any $1\leq i\leq d$, we obtain
\begin{align*}
|\alpha ^i-\beta ^i| & = |\alpha -\beta |\left| \sum_{j=0}^{i-1}\alpha ^j \beta ^{i-1-j} \right|
\leq |\alpha -\beta |\max _{0\leq j\leq i-1} |\alpha ^j \beta ^{i-1-j}| \\
& \leq |\alpha -\beta |\max _{1\leq k\leq d}(C,|\alpha |)^{k-1}.
\end{align*}
Hence, we have \eqref{eq:poly.differ}.
\end{proof}

\begin{lem}\label{lem:differ}
Let $n\geq 1$ be an integer and $\xi $ be in $\lF _q((T^{-1}))$.
Then we have
\begin{equation*}
w_n ^{*}(\xi )\leq w_n(\xi ).
\end{equation*}
\end{lem}

\begin{proof}
See Proposition 5.6 in \cite{Ooto17}.
\end{proof}

\begin{lem}\label{lem:bestrational}
Let $\xi $ be in $\lF _q ((T^{-1})), d\geq 1$ be an integer, and $\theta ,\rho ,\delta $ be positive numbers.
Assume that there exists a sequence $(p_j/q_j)_{j\geq 1}$ with $p_j ,q_j \in \lF _q[T], q_j \neq 0, \gcd(p_j, q_j)=1$ for any $j\geq 1$ such that $(|q_j|)_{j\geq 1}$ is a divergent increasing sequence, and
\begin{gather*}
\limsup _{j\rightarrow \infty } \frac{\log |q_{j+1}|}{\log |q_j|} \leq \theta ,\\
d+\delta \leq \liminf _{j\rightarrow \infty }\frac{-\log |\xi -p_j/q_j|}{\log |q_j|}, \quad \limsup _{j\rightarrow \infty }\frac{-\log |\xi -p_j/q_j|}{\log |q_j|}\leq d+\rho .
\end{gather*}
Then we have for all $1\leq n\leq d$
\begin{equation}\label{eq:b.r.a.inequ}
d-1+\delta \leq w_n ^{*}(\xi )\leq w_n(\xi )\leq \max \left( d-1+\rho , \frac{d\theta }{\delta } \right) .
\end{equation}
\end{lem}

Note that Lemma \ref{lem:bestrational} is a generalization of analogue of Lemma 1 in \cite{Amou91}.

\begin{proof}
Let $0<\iota <\delta $ be a real number.
By the assumption, there exists an integer $c_0\geq 1$ such that
\begin{equation*}
|q_j|\leq |q_{j+1}|\leq |q_j|^{\theta +\iota },\quad 
\frac{1}{|q_j|^{d+\rho +\iota }}\leq \left| \xi -\frac{p_j}{q_j} \right| \leq \frac{1}{|q_j|^{d+\delta -\iota }}
\end{equation*}
for all $j\geq c_0$.
Since $|\xi -p_j/q_j|\leq 1$ for $j\geq c_0$, we have $|q_j|\max (1,|\xi |)=\max (|p_j|,|q_j|)=H(p_j/q_j)$ for $j\geq c_0$.
Therefore, we obtain
\begin{equation*}
0< \left| \xi -\frac{p_j}{q_j} \right| \leq \frac{\max (1,|\xi |)^{d+\delta }}{H(p_j/q_j)^{d+\delta -\iota }}
\end{equation*}
for $j\geq c_0$.
Since $\iota $ is arbitrary, we have
\begin{equation*}
d-1+\delta \leq w_1 ^{*}(\xi )\leq w_2 ^{*}(\xi )\leq \ldots \leq w_d ^{*}(\xi ).
\end{equation*}
By Lemma \ref{lem:differ}, it is sufficient to show that
\begin{equation}\label{eq:remain}
w_d(\xi )\leq \max \left( d-1+\rho ,\frac{d\theta }{\delta } \right).
\end{equation}
Put $c_1:=\max (1,|\xi |)^{d-1}$.
Let $P(X) \in (\lF _q[T])[X]_{\min }$ be a polynomial of degree at most $d$ with $H(P)\geq c_1 ^{-1}|q_{c_0}|^{\frac{\delta }{\theta }}$.
We first consider the case where $P(p_j/q_j)=0$ for some $j\geq c_0$.
Then we can write $P(X)=a_j(q_j X-p_j)$ for some $a_j\in \lF _q$.
Therefore, we have
\begin{equation*}
|P(\xi )|\geq |q_j|^{-d+1-\rho -\iota }\geq H(P)^{-d+1-\rho -\iota }.
\end{equation*}
We now turn to the case where $P(q_j/p_j)\neq 0$ for all $j\geq c_0$.
We define an integer $j_0\geq c_0$ by $|q_{j_0}|\leq (c_1 H(P))^{\frac{\theta +\iota }{\delta -\iota }}<|q_{j_0+1}|$.
Then we have
\begin{equation*}
H(P)<c_1 ^{-1}|q_{j_0+1}|^{\frac{\delta -\iota }{\theta +\iota }}\leq c_1 ^{-1}|q_{j_0}|^{\delta -\iota }.
\end{equation*}
It follows from Lemma \ref{lem:poly} that
\begin{equation*}
|P(\xi )-P(p_{j_0}/q_{j_0})|\leq c_1 H(P)\left| \xi -\frac{p_{j_0}}{q_{j_0}}\right| <|q_{j_0}|^{-d}.
\end{equation*}
Since $|P(p_{j_0}/q_{j_0})|\geq |q_{j_0}|^{-d}$, we obtain
\begin{equation*}
|P(\xi )|=|P(p_{j_0}/q_{j_0})|\geq |q_{j_0}|^{-d}\geq (c_1 H(P))^{-\frac{d(\theta +\iota )}{\delta -\iota }}.
\end{equation*}
Therefore, by Lemma \ref{lem:weak}, we have
\begin{equation*}
w_d(\xi )\leq \max \left( d-1+\rho +\iota ,\frac{d(\theta +\iota )}{\delta -\iota } \right) .
\end{equation*}
Since $\iota $ is arbitrary, we obtain \eqref{eq:remain}.
\end{proof}

Let $\xi $ be in $\lF_q((T^{-1}))$ and we denote by $[a_0,a_1,a_2, \ldots ]$ the continued fraction expansion of $\xi $.
We define sequences $(p_n)_{n\geq -1}$ and $(q_n)_{n\geq -1}$ by
\begin{equation*}
\begin{cases}
p_{-1}=1,\ p_0=a_0,\ p_n=a_n p_{n-1}+p_{n-2},\ n\geq 1,\\
q_{-1}=0,\ q_0=1,\ q_n=a_n q_{n-1}+q_{n-2},\ n\geq 1.
\end{cases}
\end{equation*}
We call $(p_n/q_n)_{n\geq 0}$ the {\itshape convergent sequence} of $\xi $.
We gather fundamental properties of continued fractions in the following lemma.

\begin{lem}\label{lem:conti.lem}
Let $\xi =[a_0,a_1,a_2,\ldots ]$ be in $\lF _q((T^{-1}))$ and $(p_n/q_n)_{n\geq 0}$ be the convergent sequence of $\xi $.
Then the following hold:
\begin{enumerate}
\item $\dfrac{p_n}{q_n}=[a_0,a_1,\ldots a_n]\quad (n\geq 0)$, \label{enum:quo}
\item $\gcd(p_n,q_n)=1\quad (n\geq 0)$, \label{enum:prime}
\item $|q_n|=|a_1||a_2|\cdots |a_n|\quad (n\geq 1)$, \label{enum:q}
\item $\abs{\xi -\dfrac{p_n}{q_n}} =\dfrac{1}{\abs{q_n}\abs{q_{n+1}}}=\dfrac{1}{\abs{a_{n+1}}\abs{q_n}^2}\quad (n\geq 0)$, \label{enum:differ}
\item $\xi =\dfrac{\xi _{n+1}p_n+p_{n-1}}{\xi _{n+1}q_n+q_{n-1}}$, where $\xi _{n+1}=[a_{n+1},a_{n+2},\ldots ]\quad (n\geq 0)$. \label{enum:par}
\end{enumerate}
\end{lem}

The following lemma is a well-known result (see e.g.\ \cite{Lasjaunias00, Thakur11}).

\begin{lem}\label{lem:conti.w_1}
Let $\xi =[a_0,a_1,a_2,\ldots ]$ be in $\lF _q((T^{-1}))$ and $(p_n/q_n)_{n\geq 0}$ be the convergent sequence of $\xi $.
Then we have
\begin{equation*}
w_1 (\xi )= w_1 ^{*}(\xi )=\limsup_{n\rightarrow \infty } \frac{\deg q_{n+1}}{\deg q_n}.
\end{equation*}
\end{lem}

We extend Lemma \ref{lem:conti.w_1} by using Lemma \ref{lem:bestrational}.

\begin{prop}\label{prop:conti.w_d}
Let $d\geq 1$ be an integer and $\xi =[a_0,a_1,a_2,\ldots ]$ be in $\lF _q((T^{-1}))$.
Let $(p_n/q_n)_{n\geq 0}$ be the convergent sequence of $\xi $.
Assume that 
\begin{equation*}
\liminf_{n\rightarrow \infty }\frac{\deg q_{n+1}}{\deg q_n} \geq 2 d-1.
\end{equation*}
Then we have
\begin{equation}\label{eq:best.alg}
w_1(\xi )=w_1 ^{*}(\xi )=\ldots =w_d(\xi )=w_d ^{*}(\xi )=\limsup _{n\rightarrow \infty } \frac{\deg q_{n+1}}{\deg q_n}.
\end{equation}
\end{prop}

\begin{proof}
For $j\geq 1$, put
\begin{equation*}
A_j:=\frac{\deg q_{j+1}}{\deg q_j}=\frac{\log |q_{j+1}|}{\log |q_j|}.
\end{equation*}
It follows from Lemmas \ref{lem:differ} and \ref{lem:conti.w_1} that for all $n\geq 1$,
\begin{equation*}
\limsup _{j\rightarrow \infty } A_j \leq w_n ^{*}(\xi )\leq w_n(\xi ).
\end{equation*}
By Lemma \ref{lem:conti.lem} (\ref{enum:differ}), we have
\begin{equation*}
\frac{-\log |\xi -p_j/q_j|}{\log |q_j|}=1+A_j
\end{equation*}
for all $j\geq 1$.
It follows from Lemma \ref{lem:conti.lem} (\ref{enum:prime}) and (\ref{enum:q}) that $q_j \neq 0$ and $\gcd(p_j,q_j)=1$ for all $j\geq 1$.
Moreover, the positive integer sequence $(|q_j|)_{j\geq 1}$ is strictly increasing, which implies that it is divergent.
Applying Lemma \ref{lem:bestrational} with
\begin{equation*}
\theta =\limsup_{j\rightarrow \infty }A_j,\quad \delta =d,\quad \rho =1+\limsup_{j\rightarrow \infty }A_j-d,
\end{equation*}
we obtain
\begin{equation*}
w_n ^{*}(\xi )\leq w_n(\xi )\leq \limsup _{j\rightarrow \infty } A_j
\end{equation*}
for all $1\leq n\leq d$.
Hence, we have \eqref{eq:best.alg}.
\end{proof}

Schmidt \cite{Schmidt00} characterized Laurent series of Class IA by using continued fractions.

\begin{thm}\label{thm:classIAiff}
Let $\alpha $ be in $\lF _q((T^{-1}))$.
Then $\alpha $ is of Class IA if and only if the continued fraction expansion of $\alpha $ is of the form
\begin{equation}\label{eq:IA1}
\alpha =[a_1, a_2, \ldots ,a_t, b_1, b_2, \ldots ],
\end{equation}
where $t\geq 0$ and
\begin{equation}
b_{j+s}=
\begin{cases}\label{eq:IA2}
a b_j ^{p^k} & \text{when } j \text{ is odd}, \\
a^{-1} b_j ^{p^k} & \text{when } j \text{ is even}
\end{cases}
\end{equation}
for some $a\in \lF _q ^{\times }$, integers $s\geq 1$ and $k\geq 0$.
\end{thm}

\begin{proof}
See Theorem 4 in \cite{Schmidt00}.
\end{proof}

Thakur \cite{Thakur99} worked at the ratios of the degrees of the denominators of the convergent sequences for Laurent series of Class IA.

\begin{thm}\label{thm:supinf}
Let $\alpha \in \lF _q((T^{-1}))$ be as in \eqref{eq:IA1} and \eqref{eq:IA2}, and $(p_n/q_n)_{n\geq 0}$ be the convergent sequence of $\alpha $.
Put $d_i :=\deg b_i$ and
\begin{equation*}
r_i:=\frac{d_i}{p^k (\sum _{j=1}^{i-1}d_j )+\sum _{j=i}^{s}d_j }.
\end{equation*}
Then we have
\begin{gather}
\limsup _{n\rightarrow \infty }\frac{\deg q_{n+1}}{\deg q_n} =1+(p^k-1) \max \{ r_1, \ldots ,r_s \} ,\label{eq:limsup} \\
\liminf _{n\rightarrow \infty }\frac{\deg q_{n+1}}{\deg q_n} =1+(p^k-1) \min \{ r_1, \ldots ,r_s \} . \label{eq:liminf}
\end{gather}
\end{thm}

\begin{proof}
We have \eqref{eq:limsup} by Theorem 1 (1) in \cite{Thakur99} and Lemma \ref{lem:conti.w_1}.
Meanwhile, \eqref{eq:liminf} follows in a similar way to the proof of Theorem 1 (1) in \cite{Thakur99}.
\end{proof}

We define a valuation $v$ on $\lF_q((T^{-1}))$ by $v(\xi )=-\log _q\abs{\xi }$ for $\xi \in \lF_q ((T^{-1}))$.

\begin{lem}\label{lem:irrcri}
	Let
	\begin{equation*}
	P(X)=X^m+\sum_{i=1}^{m}a_iX^{m-i}
	\end{equation*}
	be a monic polynomial in $(\lF_q((T^{-1})))[X]$.
	Assume that $v(a_m)$ is a positive integer with $\gcd (v(a_m),m)=1$.
	Then $P(X)$ is irreducible over $\lF_q((T^{-1}))$ if and only if $v(a_i)/i>v(a_m)/m$ for all $1\leq i\leq m-1$.
\end{lem}

\begin{proof}
	See Proposition 2.2 in \cite{Popescu95}.
\end{proof}

We give a sufficient condition to determine degrees of Laurent series of Class IA.
The following lemma is inspired by the proof of Theorem 2 in \cite{Chen13}.

\begin{lem}\label{lem:classiadegree}
	Let $\alpha \in \lF_q((T^{-1}))$ be as in \eqref{eq:IA1} and \eqref{eq:IA2}.
	If $\gcd (\deg b_s,p)=1$, then $\alpha $ is algebraic of degree $p^k+1$.
\end{lem}

\begin{proof}
By Lemma \ref{lem:conti.lem} \eqref{enum:par}, it is sufficient to show that $\beta =[b_1,b_2,\ldots ]$ is algebraic of degree $p^k+1$.	
Let $(p_n/q_n)_{n\geq 0}$ be the convergent sequence of $\beta $.
Since
\begin{equation*}
\beta _s=[ab_1^{p^k},a^{-1}b_2^{p^k},\ldots ]=a\beta ^{p^k},
\end{equation*}
it follows from Lemma \ref{lem:conti.lem} \eqref{enum:par} that $\beta $ is a root of the monic polynomial
\begin{equation*}
P(X):=X^{p^k+1}-\frac{p_{s-1}}{q_{s-1}}X^{p^k}+\frac{q_{s-2}}{aq_{s-1}}X-\frac{p_{s-2}}{aq_{s-1}}.
\end{equation*}
Put
\begin{equation*}
Q(X):=X^{p^k}+\left( \beta -\frac{p_{s-1}}{q_{s-1}}\right)\sum_{j=0}^{p^k-1}\beta ^{p^k-1-j}X^j +\frac{q_{s-2}}{aq_{s-1}},
\end{equation*}
then we have $P(X)=(X-\beta )Q(X)$.
For $1\leq i\leq p^k$, let $c_i$ be the coefficient of $X^{p^k-i}$ in $Q(X)$.
By Lemma \ref{lem:conti.lem} \eqref{enum:q} and \eqref{enum:differ}, we deduce that
\begin{gather*}
v(c_i)=(p^k-i+1)\deg b_1+2\sum_{j=2}^{s}\deg b_j\quad (1\leq i\leq p^k-1),\\
v(c_{p^k})=\deg b_s.
\end{gather*}
Therefore, by Lemma \ref{lem:irrcri}, the monic polynomial $Q(X)$ is irreducible over $\lF_q((T^{-1}))$.
Since $\beta \notin \lF_q(T)$, we derive that $P(X)$ is irreducible over $\lF_q(T)$.
\end{proof}

\begin{lem}\label{lem:seq}
Let $d\geq 1$ be an integer and $w>2d-1$ be a rational number.
Write $w=a/b,$ where $a,b\geq 1$ are integers and $a=p^m a',$ where $a'\geq 1, m\geq 0$ are integers and $\gcd(a',p)=1$.
Then there exist a strictly increasing sequence of  positive integers $(k_j)_{j\geq 1}$, a sequence of integers $(n_j)_{j\geq 1}$, and a sequence of rational numbers $(u_j)_{j\geq 1}$ such that for any $j\geq 1, n_j\geq 3, u_j=a_j/b_j,$ where $a_j,b_j \geq 1$ are integers, $ \gcd(a_j,p)=1, p^m | b_j,$ and
\begin{equation}\label{eq:minmax}
2 d-1< \min \left\{ w, u_j , \frac{p^{k_j}}{w u _j ^{n_j-2}} \right\} ,\quad \max \left\{ w, u_j , \frac{p^{k_j}}{w u _j ^{n_j-2}} \right\} =w.
\end{equation}
\end{lem}

\begin{proof}
The proof is by induction on $j$.
By assumption, we take an integer $n_1\geq 3$ with $(w/(2d-1))^{n_1-1}> p$.
Then we have $\log _p w^{n_1}-\log _p w(2d-1)^{n_1-1}>1$.
This implies that there exists an integer $k_1\geq 1$ such that
\begin{equation*}
w (2 d-1)^{n_1-1}<p^{k_1}<w^{n_1}.
\end{equation*}
Then we have
\begin{equation*}
\max \left\{ 2 d-1, \left( \frac{p^{k_1}}{w^2}\right) ^{\frac{1}{n_1-2}}\right\} <\min \left\{ w,\left( \frac{p^{k_1}}{(2 d-1)w} \right) ^{\frac{1}{n_1-2}} \right\} .
\end{equation*}
Let $r\geq 2$ be an integer such that $\gcd(r,p)=1$.
By Lemma 2.5.9 in \cite{Allouche03}, the set $\{ r^y/p^x \mid x,y\in \lZ _{\geq 0} \}$ is dense in $\lR_{>0}$.
Therefore, we can take a rational number $u_1=a_1/b_1$ such that $a_1,b_1 \in \lZ _{>0}, \gcd (a_1,p)=1, p^m | b_1$, and
\begin{equation*}
\max \left\{ 2 d-1, \left( \frac{p^{k_1}}{w^2}\right) ^{\frac{1}{n_1-2}} \right\} <u_1<\min \left\{ w,\left( \frac{p^{k_1}}{(2 d-1)w} \right) ^{\frac{1}{n_1-2}} \right\} .
\end{equation*}
Thus, we have \eqref{eq:minmax} when $j=1$.
Assume that \eqref{eq:minmax} holds for $j=1,\ldots ,i$.
We take an integer $n_{i+1}\geq 3$ with with $(w/(2d-1))^{n_{i+1}-1}> p^{k_i}$.
In a similar way to the above proof, we can take an integer $k_{i+1}$ with $k_i<k_{i+1}$ and a rational number $u_{i+1}$ which satisfy \eqref{eq:minmax}.
This completes the proof.
\end{proof}

Quadratic Laurent series are characterized by continued fractions as follows:

\begin{lem}\label{lem:ult}
Let $\xi =[a_0,a_1,\ldots ]$ be in $\lF _q((T^{-1}))$.
Then $\xi $ is quadratic if and only if $(a_n)_{n\geq 0}$ is ultimately periodic.
\end{lem}

\begin{proof}
See e.g.\ Th\'{e}or\`{e}me 4 in \cite[CHAPITRE IV]{Mathan70} or Theorems 3 and 4 in \cite{Chaichana06}.
\end{proof}

The following lemma is well-known and immediately seen.

\begin{lem}\label{lem:height}
Let $P(X)$ be in $(\lF _q[T])[X]$.
Assume that $P(X)$ can be factorized as
\begin{equation*}
P(X)=A\prod_{i=1}^{n} (X-\alpha _i),
\end{equation*}
where $A\in \lF _q[T]$ and $\alpha _i \in \overline{\lF _q(T)}$ for $1\leq i\leq n$.
Then we have
\begin{align}\label{eq:heighteq}
H(P)=|A|\prod_{i=1}^{n} \max (1, |\alpha _i|).
\end{align}
Furthermore, for $P(X), Q(X) \in (\lF _q[T])[X]$, we have
\begin{align}\label{eq:heighttwo}
H(P Q)=H(P)H(Q).
\end{align}
\end{lem}

Let $\alpha \in \overline{\lF _q(T)}$ be a quadratic number.
If $\insep \alpha =1$, let $\alpha ' \neq \alpha $ be the Galois conjugate of $\alpha $.
If $\insep \alpha =2$, let $\alpha ' =\alpha $.

\begin{lem}\label{lem:Galois}
Let $\alpha \in \overline{\lF _q(T)}$ be a quadratic number.
If $\alpha \neq \alpha '$, then we have
\begin{equation*}
|\alpha -\alpha '|\geq H(\alpha )^{-1}.
\end{equation*}
\end{lem}

\begin{proof}
This is clear by using the discriminant of the minimal polynomial of $\alpha $ (see e.g.\ \cite[Appendix A]{Cassels86} for the definition and properties of the discriminant).
We refer to Lemma 3.5 in \cite{Ooto17} for a direct proof.
\end{proof}

We recall the Liouville inequalities for Laurent series over a finite field.

\begin{lem}\label{lem:Lio.inequ1}
Let $ P(X) \in (\lF_q[T])[X]$ be a non-constant polynomial of degree $m$ and $\alpha \in \overline{\lF _q(T)}$ be a number of degree $n$.
Assume that $P(\alpha )\not= 0$.
Then we have
\begin{equation*}
|P(\alpha )|\geq H(P)^{-n+1} H(\alpha )^{-m}.
\end{equation*}
\end{lem}

\begin{proof}
See e.g.\ Lemma 4 in \cite{Muller93} or Proposition 3.2 in \cite{Ooto17}.
\end{proof}

\begin{lem}\label{lem:Lio.inequ2}
Let $\alpha ,\beta \in \overline{\lF _q(T)}$ be distinct numbers of degree $m$ and $n$, respectively.
Then we have
\begin{equation*}
|\alpha -\beta |\geq H(\alpha )^{-n} H(\beta )^{-m}.
\end{equation*}
\end{lem}

\begin{proof}
See e.g.\ Korollar 3 in \cite{Muller93} or Proposition 3.4 in \cite{Ooto17}.
\end{proof}

The lemma below is an immediate consequence of Lemmas \ref{lem:Lio.inequ1} and \ref{lem:Lio.inequ2}.

\begin{lem}\label{lem:alg}
Let $n\geq 1$ be an integer and $\xi $ be an algebraic Laurent series of degree $d$.
Then we have
\begin{equation*}
w_n(\xi ), w_n ^{*}(\xi )\leq d-1.
\end{equation*}
\end{lem}

We give a key lemma for the proof of Theorems \ref{thm:mainconti1} and \ref{thm:mainconti2} as follows:

\begin{lem}\label{lem:bestquad}
Let $d\geq 2$ be an integer.
Let $\xi $ be in $\lF_q((T^{-1}))$ and $\theta ,\rho ,\delta $ be positive numbers.
Assume that there exists a sequence $(\alpha _j)_{j\geq 1}$ such that for any $j\geq 1$, $\alpha _j\in \overline{\lF _q(T)}$ is quadratic, $(H(\alpha _j))_{j\geq 1}$ is a divergent increasing sequence, and
\begin{gather*}
\limsup_{j\rightarrow \infty } \frac{\log H(\alpha _{j+1})}{\log H(\alpha _j)}\leq \theta ,\\
d+\delta \leq \liminf_{j\rightarrow \infty } \frac{-\log |\xi -\alpha _j|}{\log H(\alpha _j)},\quad \limsup_{j\rightarrow \infty } \frac{-\log |\xi -\alpha _j|}{\log H(\alpha _j)}\leq d+\rho .
\end{gather*}
If $2d\theta \leq (d-2+\rho )\delta $, then we have for all $2\leq n\leq d$,
\begin{equation}\label{eq:w_2;1}
d-1+\delta \leq w_n ^{*}(\xi )\leq d-1+\rho .
\end{equation}
Furthermore, assume that there exist a non-negative number $\varepsilon $ and a positive number $c$ such that for any $j\geq 1, 0<|\alpha _j-\alpha _j '|\leq c$ and
\begin{equation}\label{eq:ipu}
\limsup_{j\rightarrow \infty }\frac{-\log |\alpha _j -\alpha _j '|}{\log H(\alpha _j)}\geq \varepsilon .
\end{equation}
If $2d\theta \leq (d-2+\delta )\delta $, then we have for all $2\leq n\leq d$,
\begin{equation}\label{eq:w_2;2}
d-1+\delta \leq w_n ^{*}(\xi )\leq d-1+\rho ,\quad \varepsilon \leq w_n(\xi )-w_n ^{*}(\xi ).
\end{equation}
Finally, assume that there exists a non-negative number $\chi $ such that
\begin{equation}\label{eq:kai}
\limsup_{i\rightarrow \infty} \frac{-\log |\alpha _i-\alpha _i '|}{\log H(\alpha _i)}\leq \chi .
\end{equation}
Then we have for all $2\leq n\leq d$,
\begin{equation}\label{eq:w_2;3}
d-1+\delta \leq w_n ^{*}(\xi )\leq d-1+\rho ,\quad \varepsilon \leq w_n(\xi )-w_n ^{*}(\xi )\leq \chi .
\end{equation}
\end{lem}

\begin{proof}
Let $0<\iota <\delta $ be a real number.
Then there exists a positive integer $c_0$ such that
\begin{equation*}
\frac{1}{H(\alpha _j)^{d+\rho +\iota }} \leq |\xi -\alpha _j|\leq \frac{1}{H(\alpha _j)^{d+\delta -\iota }},\quad H(\alpha _j)\leq H(\alpha _{j+1})\leq H(\alpha _j)^{\theta +\iota }
\end{equation*}
for all $j\geq c_0$.
Since $\iota $ is arbitrary, we have $w_2 ^{*}(\xi )\geq d-1+\delta $.
Let $\alpha \in \overline{\lF _q(T)} \setminus \{\alpha _j \mid j\geq 1 \}$ be an algebraic number of degree at most $d$ with $H(\alpha )\geq H(\alpha _{c_0})^{\frac{\delta }{2 \theta }}$.

Assume that $2d\theta \leq (d-2+\rho )\delta $.
We define an integer $j_0\geq c_0$ such that $H(\alpha _{j_0})\leq H(\alpha )^{\frac{2(\theta +\iota )}{\delta -\iota }}<H(\alpha _{j_0 +1})$.
Since
\begin{equation*}
H(\alpha )<H(\alpha _{j_0 +1})^{\frac{\delta -\iota }{2(\theta +\iota )}}\leq H(\alpha _{j_0})^{\frac{\delta -\iota }{2}},
\end{equation*}
we obtain
\begin{equation*}
|\alpha -\alpha _{j_0}|\geq H(\alpha )^{-2}H(\alpha _{j_0})^{-d}>H(\alpha _{j_0})^{-d-\delta +\iota }\geq |\xi -\alpha _{j_0}|
\end{equation*}
by Lemma \ref{lem:Lio.inequ2}.
We derive that
\begin{equation}\label{eq:llow}
|\xi -\alpha |=|\alpha -\alpha _{j_0}|\geq H(\alpha )^{-2}H(\alpha _{j_0})^{-d}\geq H(\alpha )^{-2-\frac{2 d(\theta +\iota )}{\delta -\iota }},
\end{equation}
which implies
\begin{equation*}
w_d ^{*}(\xi )\leq \max \left( d-1+\rho +\iota , 1+\frac{2 d(\theta +\iota )}{\delta -\iota } \right) .
\end{equation*}
Since $\iota $ is arbitrarily small, \eqref{eq:w_2;1} follows.

Next, we assume that $2d\theta \leq (d-2+\delta )\delta $ and there exist a non-negative number $\varepsilon $ and a positive number $c$ such that for any $j\geq 1, 0<|\alpha _j-\alpha _j '|\leq c$ and \eqref{eq:ipu} holds.
Since $\delta \leq \rho $, we have \eqref{eq:w_2;1}.
By the assumption and \eqref{eq:llow}, the sequence $(\alpha _j)_{j\geq 1}$ is the best approximation to $\xi $ of degree at most $d$, that is,
\begin{equation}\label{eq:llow2}
w_d ^{*}(\xi )=\limsup _{j\rightarrow \infty }\frac{-\log |\xi -\alpha _j|}{\log H(\alpha _j)}-1.
\end{equation}
Therefore, we have $w_2 ^{*}(\xi )=\ldots =w_d ^{*}(\xi )$.
In what follows, we show that $\varepsilon \leq w_n(\xi )-w_n ^{*}(\xi )$ for all $2\leq n\leq d$.
For any $j\geq 1$, we denote by $P_j(X)=A_j(X-\alpha _j)(X-\alpha _j ')$ the minimal polynomial of $\alpha _j$.
Since $|\xi -\alpha _j|\leq 1$ and $|\alpha _j-\alpha _j '|\leq c$ for $j\geq c_0$, we have
\begin{equation*}
\max (1,|\xi |)\asymp \max (1,|\alpha _j|)\asymp \max (1,|\alpha _j '|)
\end{equation*}
for $j\geq c_0$.
It follows from Lemma \ref{lem:height} that $H(P_j)\asymp |A_j|$ for $j\geq c_0$.
By Lemma \ref{lem:Galois}, we have $|\xi -\alpha _j|<|\alpha _j-\alpha _j '|$ for $j\geq c_0$, which implies $|\xi -\alpha _j '|=|\alpha _j -\alpha _j '|$ for $j\geq c_0$.
Hence, we have
\begin{equation}\label{eq:llow3}
|P_j (\xi )|\asymp H(P_j)|\xi -\alpha _j||\alpha _j-\alpha _j '|
\end{equation} 
for $j\geq c_0$.
It follows from \eqref{eq:llow2} and \eqref{eq:llow3} that $w_d ^{*}(\xi )+\varepsilon \leq w_2(\xi )$.
Thus, we have $\varepsilon \leq w_n(\xi )-w_n ^{*}(\xi )$ for all $2\leq n\leq d$.

Finally, we assume \eqref{eq:kai}.
By \eqref{eq:llow3}, we obtain
\begin{equation*}
\limsup _{j\rightarrow \infty} \frac{-\log |P_j(\xi )|}{\log H(P_j)}\leq w_2 ^{*}(\xi )+\chi .
\end{equation*}
Recall that $C_1(\alpha ,C,d)$ is defined in Lemma \ref{lem:poly}.
Put $c_1:=\max _{1\leq i\leq d}C_1(\xi ,1,i)$.
Let $P(X)\in (\lF _q[T])[X]_{\min }$ be a polynomial of degree at most $d$ with $H(P)\geq c_1 ^{-\frac{1}{2}}H(\alpha _{c_0})^{\frac{\delta }{2\theta }}$ and $P(\alpha _j)\neq 0$ for all $j\geq 1$.
We define an integer $j_1\geq c_0$ such that $H(\alpha _{j_1})\leq (c_1 H(P)^2)^{\frac{\theta +\iota }{\delta -\iota }}<H(\alpha _{j_1+1})$.
Since
\begin{equation*}
H(P)<c_1 ^{-\frac{1}{2}}H(\alpha _{j_1+1})^{\frac{\delta -\iota }{2(\theta +\iota )}}\leq c_1 ^{-\frac{1}{2}}H(\alpha _{j_1})^{\frac{\delta -\iota }{2}},
\end{equation*} 
we have
\begin{equation*}
|P(\xi )-P(\alpha _{j_1})|
\leq c_1 H(P)|\xi -\alpha _{j_1}|
< H(\alpha _{j_1})^{-d}H(P)^{-1}
\leq |P(\alpha _{j_1})|
\end{equation*}
by Lemmas \ref{lem:poly} and \ref{lem:Lio.inequ1}.
Therefore, we obtain
\begin{equation*}
|P(\xi )|=|P(\alpha _{j_1})|\geq H(\alpha _{j_1})^{-d}H(P)^{-1} \geq c_1 ^{-\frac{d(\theta +\iota )}{\delta -\iota }}H(P)^{-1-\frac{2 d(\theta +\iota )}{\delta -\iota }}.
\end{equation*}
Hence, we get 
\begin{equation*}
w_d(\xi )\leq \max \left( w_2 ^{*}(\xi )+\chi , 1+\frac{2 d(\theta +\iota )}{\delta -\iota } \right) 
\end{equation*}
by Lemma \ref{lem:weak}.
Since $\iota $ is arbitrarily small, we have $w_d(\xi )\leq w_2 ^{*}(\xi )+\chi $.
This completes the proof.
\end{proof}

\section{Proof of main results}\label{sec:mainproof}

\begin{proof}[Proof of Theorem \ref{thm:mainalg}]
We take a strictly increasing sequence of  positive integers $(k_j)_{j\geq 1}$, a sequence of integers $(n_j)_{j\geq 1}$, and a sequence of rational numbers $(u_j)_{j\geq 1}$ as in Lemma \ref{lem:seq}.
For $j\geq 1$, we put
\begin{gather*}
d_{1,j}:=\frac{b_j ^{n_j-2}(a-b)}{p^m}, \quad d_{i,j}:=\frac{a a_j ^{i-2}b_j ^{n_j-i-1}(a_j-b_j)}{p^m} \quad (2\leq i\leq n_j-1), \\
d_{n_j,j}:=\frac{p^{k_j}b b_j ^{n_j-2}-a a_j ^{n_j-2}}{p^m}.
\end{gather*}
Then we have $d_{i,j}\in \lZ _{>0}$ and $\gcd(d_{n_j,j},p)=1$ for all $j\geq 1, 1\leq i\leq n_j$.
Now we take polynomials $A_{1,j},\ldots ,A_{n_j,j}\in \lF _q[T]$ with $\deg A_{i,j}=d_{i,j}$.
Put
\begin{equation*}
\xi _j:=[A_{1,j},\ldots ,A_{n_j,j},A_{1,j}^{p^{k_j}},\ldots ,A_{n_j,j}^{p^{k_j}},A_{1,j}^{p^{2 k_j}},\ldots ]\in \lF _q((T^{-1}))
\end{equation*}
and let $(p_{n,j}/q_{n,j})_{n\geq 0}$ be the convergent sequence of $\xi _j$.
By Theorem \ref{thm:classIAiff}, $\xi _j$ is of Class IA.
Therefore, by Lemma \ref{lem:classiadegree}, we deduce that $\xi _j$ is algebraic of degree $p^{k_j}+1$.
For $1\leq i\leq n_j$, we put
\begin{equation*}
r_{i,j}:=\frac{d_{i,j}}{p^{k_j}(\sum _{\ell =1}^{i-1}d_{\ell ,j})+\sum _{\ell =i}^{n_j}d_{\ell ,j}}.
\end{equation*}
Then a straightforward computation shows that
\begin{gather*}
r_{1,j}=\frac{a-b}{(p^{k_j}-1)b},\quad r_{i,j}=\frac{a_j-b_j}{(p^{k_j}-1)b_j}\quad (2\leq i\leq n_j-1), \\
r_{n_j,j}=\frac{p^{k_j}b b_j ^{n_j-2}-a a_j ^{n_j-2}}{(p^{k_j}-1)a a_j ^{n_j-2}}.
\end{gather*}
By Theorem \ref{thm:supinf} and Lemma \ref{lem:seq}, we obtain
\begin{equation*}
\limsup _{n\rightarrow \infty }\frac{\deg q_{n+1,j}}{\deg q_{n,j}}=w,\quad \liminf _{n\rightarrow \infty }\frac{\deg q_{n+1,j}}{\deg q_{n,j}}>2 d-1.
\end{equation*}
It follows from Proposition \ref{prop:conti.w_d} that
\begin{equation*}
w_1(\xi _j)=w_1 ^{*}(\xi _j)=\ldots =w_d(\xi _j)=w_d ^{*}(\xi _j)=w.
\end{equation*}
This completes the proof. 
\end{proof}

\begin{proof}[Proof of Theorem \ref{thm:realequal}]
Let $(\varepsilon _n)_{n\geq 1}$ be a sequence over the set $\{ 0,1 \}$.
We define recursively the sequences $(a_n)_{n\geq 0}, (P_n)_{n\geq -1}$, and $(Q_n)_{n\geq -1}$ by 
\begin{equation*}
\begin{cases}
a_0=0,\ a_1=T+\varepsilon _1,\ a_n=T^{\lfloor (w-1)\deg Q_{n-1}\rfloor}+\varepsilon _n,\ n\geq 2,\\
P_{-1}=1,\ P_0=0,\ P_n=a_n P_{n-1}+P_{n-2},\ n\geq 1,\\
Q_{-1}=0,\ Q_0=1,\ Q_n=a_n Q_{n-1}+Q_{n-2},\ n\geq 1.
\end{cases}
\end{equation*}
Set $\xi _w:=[0,a_1,a_2,\ldots ]$.
Then $(P_n/Q_n)_{n\geq 0}$ is the convergent sequence of $\xi _w$.
It follows from Lemma \ref{lem:conti.lem} (\ref{enum:q}) that
\begin{equation*}
\lim _{n\rightarrow \infty } \frac{\deg Q_{n+1}}{\deg Q_n}=w.
\end{equation*}
Therefore, by Proposition \ref{prop:conti.w_d}, we have
\begin{equation*}
w_1(\xi _w)=w_1 ^{*}(\xi _w)=\ldots =w_d(\xi _w)=w_d ^{*}(\xi _w)=w.
\end{equation*}
\end{proof}

\begin{proof}[Proof of Theorem \ref{thm:mainconti1}]
For any $j\geq 1$, we put
\begin{equation*}
\xi _{w,j}:=[0,a_{1,w},\ldots ,a_{\lfloor w^j\rfloor ,w},\overline{a}].
\end{equation*}
Then $\xi _{w,j}$ is quadratic by Lemma \ref{lem:ult}.
It follows from the proof of Theorem 1.1 in \cite{Ooto17} that $(H(\xi _{w,j}))_{j\geq 1}$ is a divergent increasing sequence, and
\begin{gather*}
\lim _{j\rightarrow \infty }\frac{-\log |\xi _w -\xi _{w,j}|}{\log H(\xi _{w,j})}=w,\quad \lim _{j\rightarrow \infty }\frac{-\log |\xi _{w,j} -\xi _{w,j} '|}{\log H(\xi _{w,j})}=1,\\
\limsup _{j\rightarrow \infty }\frac{\log H(\xi _{w,j+1})}{\log H(\xi _{w,j})}\leq w.
\end{gather*}
By the definition of $w$, we have $2dw\leq (w-2)(w-d)$.
Applying Lemma \ref{lem:bestquad} with $\delta =\rho =w-d, \varepsilon =\chi =1$, and $\theta =w$, we obtain \eqref{eq:mainconti1} for all $2\leq n\leq d$.
\end{proof}

\begin{proof}[Proof of Theorem \ref{thm:mainconti2}]
For any $j\geq 1$, we put
\begin{equation*}
\xi _{w,\eta, j}:=[0,a_{1,w,\eta },\ldots ,a_{\lfloor w^j\rfloor ,w,\eta },\overline{a,\ldots ,a,c}],
\end{equation*}
where the length of the periodic part is $\lfloor \eta w^j\rfloor $.
Then $\xi _{w,\eta ,j}$ is quadratic by Lemma \ref{lem:ult}.
It follows from the proof of Theorem 1.2 in \cite{Ooto17} that $(H(\xi _{w,\eta ,j}))_{j\geq 1}$ is a divergent increasing sequence, and
\begin{gather*}
\lim _{j\rightarrow \infty }\frac{-\log |\xi _{w,\eta } -\xi _{w,\eta ,j}|}{\log H(\xi _{w,\eta ,j})}=\frac{2 w}{2+\eta },\quad \lim _{j\rightarrow \infty }\frac{-\log |\xi _{w,\eta ,j} -\xi _{w,\eta ,j} '|}{\log H(\xi _{w,\eta ,j})}=\frac{2}{2+\eta },\\
\limsup _{j\rightarrow \infty }\frac{\log H(\xi _{w,\eta ,j+1})}{\log H(\xi _{w,\eta ,j})}\leq w.
\end{gather*}
A direct computation shows that
\begin{equation*}
2 d w\leq \left(\frac{2 w}{2+\eta }-2 \right) \left( \frac{2 w}{2+\eta }-d \right) .
\end{equation*}
Applying Lemma \ref{lem:bestquad} with
\begin{equation*}
\delta =\rho =\frac{2 w}{2+\eta }-d,\quad \varepsilon =\chi =\frac{2}{2+\eta },\quad \theta =w,
\end{equation*}
we have \eqref{eq:mainconti2} for all $2\leq n\leq d$.
\end{proof}

\section{Further remarks}\label{sec:conrem}

In this section, we give some theorems associated to the main results.

\subsection{Relationship between automatic sequences and Diophantine exponents}

Let $k\geq 2$ be an integer.
We denote by $\Sigma _k$ the set $\{ 0,1,\ldots ,k-1 \}$.
A $k$-{\itshape automaton} is defined to be a sextuple
\begin{equation*}
A=(Q, \Sigma _k, \delta , q_0, \Delta ,\tau ),
\end{equation*} 
where $Q$ is a finite set of {\itshape states}, $\delta :Q\times \Sigma _k\rightarrow Q$ is a {\itshape transition function}, $q_0 \in Q$ is an {\itshape initial state}, a finite set $\Delta $ is an {\itshape output alphabet}, and $\tau :Q\rightarrow \Delta $ is an {\itshape output function}.
For $q\in Q$ and a finite word $W=w_0 w_1 \cdots w_n$ over $\Sigma _k$, we define $\delta (q,W)$ recursively by $\delta (q,W)=\delta (\delta (q,w_0 w_1\cdots w_{n-1}), w_n)$.
For an integer $n\geq 0$, we put $W_n:=w_r w_{r-1}\cdots w_0$, where $\sum _{i=0}^{r}w_i k^i$ is the $k$-ary expansion of $n$.
An infinite sequence $(a_n)_{n\geq 0}$ is said to be $k$-{\itshape automatic} if there exists a $k$-automaton $A=(Q, \Sigma _k, \delta , q_0, \Delta ,\tau )$ such that $a_n=\tau (\delta (q_0,W_n))$ for all $n\geq 0$.

Christol, Kamae, Mendes France, and Rauzy \cite{Christol80} characterized algebraic Laurent series by using finite automatons.
More precisely, they showed that for a sequence $(a_n)_{n\geq 0}$ over $\lF _q$, the Laurent series $\sum _{n=0}^{\infty }a_n T^{-n}$ is algebraic if and only if $(a_n)_{n\geq 0}$ is $p$-automatic.
It is known that for $m\geq 1$ and $k\geq 2$, a sequence $(a_n)_{n\geq 0}$ is $k$-automatic if and only if it is $k^m$-automatic (see Theorem 6.6.4 in \cite{Allouche03}).
Therefore, we obtain the following corollary of Theorem \ref{thm:mainalg}.

\begin{cor}\label{cor:pauto}
Let $d,m\geq 1$ be integers and $w>2d-1$ be a rational number.
Then there exists a $p^m$-automatic sequence $(a_n)_{n\geq 0}$ over $\lF _q$ such that
\begin{equation*}
w_1(\xi )=w_1 ^{*}(\xi )=\ldots =w_d(\xi )=w_d ^{*}(\xi )=w,
\end{equation*}
where $\xi =\sum _{n=0}^{\infty }a_n T^{-n}$.
\end{cor}

In this subsection, we consider the problem of determining whether or not we can extend Corollary \ref{cor:pauto} to $k$-automatic sequence for any integer $k\geq 2$.
It is the natural problem in view of Corollary \ref{cor:pauto}.

Let $k\geq 2$ be an integer.
We define a set $S_k$ of rational numbers as follows:
\begin{equation*}
S_k=\left\{ \frac{k^a}{\ell }\setmid a, \ell \in \lZ _{\geq 1}\right\} .
\end{equation*}

Bugeaud \cite{Bugeaud08} proved that for an integer $k\geq 2$ and $w\in S_k$ with $w>2$, there exists a $k$-automatic sequence $(a_n)_{n\geq 0}$ over $\{ 0,2\}$ such that $w_1(\xi )=w-1$, where $\xi =\sum _{n=0}^{\infty }a_n/3^n$.
The proof of this result essentially depends on the Folding Lemma and an analogue of Lemma \ref{lem:conti.w_1}.
It is known that the Folding Lemma holds for Laurent series over a finite field. 
For the statement and proof of the Folding Lemma, we refer the readers to \cite[Proposition 2]{Poorten92} and \cite[the proof of Theorem 1]{Shallit79}.
We have the following theorem which is similar to Bugeaud's result.

\begin{thm}
Let $k\geq 2$ be an integer and $w>2$ be in $S_k$.
Then there exists a $k$-automatic sequence $(a_n)_{n\geq 0}$ over $\lF _q$ such that $w_1(\xi )=w-1$, where $\xi =\sum _{n=0}^{\infty }a_n T^{-n}$.
\end{thm}

Using Lemma \ref{lem:bestrational}, we prove the following theorem.

\begin{thm}
Let $d\geq 1$ be an integer, $k\geq 2$ be an integer, and $w>(2d+1+\sqrt{4d^2+1})/2$ be in $S_k$.
Then there exists a $k$-automatic sequence $(a_n)_{n\geq 0}$ over $\lF _q$ such that 
\begin{equation*}
w_1(\xi )=w_1 ^{*}(\xi )=\ldots =w_d(\xi )=w_d^{*}(\xi )=w-1,
\end{equation*}
where $\xi =\sum _{n=0}^{\infty }a_n T^{-n}$.
\end{thm}

Note that $(2d-1+\sqrt{4d^2+1})/2$ is greater than $2d-1$ for any $d\geq 1$.

\begin{proof}
Slightly modifying the proof of Theorem 1.2 in \cite{Bugeaud152}, we deduce that there exists a $k$-automatic sequence $(a_n)_{n\geq 0}$ with $a_n \in \{ 0,1\}$ for all $n\geq 1$ which satisfies the following properties:
\begin{itemize}
	\item $\displaystyle \frac{2d+1+\sqrt{4d^2+1}}{2}<\frac{n_{j+1}}{n_j}\leq w$ holds for all $j\geq 0$,
	\item there exist infinitely many $j\geq 0$ such that $\displaystyle \frac{n_{j+1}}{n_j}=w$,
\end{itemize}
where 
\begin{equation*}
\{ n\in \lZ _{\geq 0}\mid a_n=1\} =:\{ n_0<n_1<n_2<\ldots \} .
\end{equation*}

We put $q_j:=T^{n_j}$ and $p_j:=1+T^{n_j-n_{j-1}}+\cdots +T^{n_j-n_0}$ for any $j\geq 0$.
Then we have for any $j\geq 0$, $\gcd(p_j, q_j)=1$ and $p_j/q_j=\sum _{n=0}^{n_j}a_n T^{-n}.$
We put $\xi :=\sum _{n=0}^{\infty }a_n T^{-n}.$
A direct computation shows that
\begin{equation*}
\frac{-\log |\xi -p_{n_j}/q_{n_j}|}{\log |q_{n_j}|}=\frac{\log |q_{n_{j+1}}|}{\log |q_{n_j}|}=\frac{n_{j+1}}{n_j}
\end{equation*}
for any $j\geq 0$.
By the definition of $w_1$, we obtain $w_1(\xi )\geq w-1$.
Applying Lemma \ref{lem:bestrational} with $\theta =w, \rho =w-d$, and $\delta =(1+\sqrt{4d^2+1})/2$, we deduce that
\begin{equation*}
w_1(\xi )=w_1 ^{*}(\xi )=\ldots =w_d(\xi )=w_d ^{*}(\xi )=w-1.
\end{equation*}
\end{proof}

\subsection{Analogues of Theorems \ref{thm:mainconti1} and \ref{thm:mainconti2} for real and $p$-adic numbers}

In this subsection, we give analogues of Theorems \ref{thm:mainconti1} and \ref{thm:mainconti2}, which are generalizations of Theorems 4.1 and 4.2 in \cite{Bugeaud12}, and Theorems 1 and 2 in \cite{Bugeaud15}.

\begin{thm}\label{thm:rmainconti1}
Let $d\geq 2$ be an integer and $w\geq (3d+2+\sqrt{9d^2+4d+4})/2$ be a real number.
Let $a,b$ be distinct positive integers.
Let $(a_{n,w})_{n\geq 1}$ be a sequence given by
\begin{equation*}
a_{n,w}=\begin{cases}
b & \text{if } n=\lfloor w^i \rfloor \text{ for some integer } i\geq 0,\\
a & \text{otherwise}.
\end{cases}
\end{equation*}
Set the continued fraction $\xi _w:=[0,a_{1,w},a_{2,w},\ldots ]\in \lR $.
Then we have 
\begin{equation}\label{eq:rmainconti1}
w_n ^{*}(\xi _w)=w-1,\quad w_n(\xi _w)=w 
\end{equation}
for all $2\leq n\leq d$.
\end{thm}

\begin{thm}\label{thm:rmainconti2}
Let $d\geq 2$ be an integer, $w\geq 121d^2$ be a real number, and $a,b,c$ be distinct positive integers.
Let $0<\eta <\sqrt{w}/d$ be a positive number and put $m_i:=\lfloor (\lfloor w^{i+1}\rfloor -\lfloor w^i-1\rfloor )/\lfloor \eta w^i\rfloor \rfloor $ for all $i\geq 1$.
Let $(a_{n,w,\eta })_{n\geq 1}$ be the sequence given by
\begin{equation*}
a_{n,w,\eta }=\begin{cases}
b & \text{if } n=\lfloor w^i \rfloor \text{ for some integer } i\geq 0,\\
c & \parbox{250pt}{$\text{if } n \neq \lfloor w^i\rfloor \text{ for all integers }i\geq 0, \text{and } n=\lfloor w^j\rfloor +m \lfloor \eta w^j\rfloor \text{ for some integers }1\leq m\leq m_j, j\geq 1,$}\\
a & \text{otherwise}.
\end{cases}
\end{equation*}
Set the continued fraction $\xi _{w,\eta }:=[0,a_{1,w,\eta },a_{2,w,\eta },\ldots ]\in \lR $.
Then we have
\begin{equation}\label{eq:rmainconti2}
w_n ^{*}(\xi _{w,\eta })=\frac{2 w-2-\eta }{2+\eta },\quad w_n(\xi _{w,\eta })=\frac{2 w-\eta }{2+\eta }
\end{equation}
for all $2\leq n\leq d$.
Hence, we have
\begin{equation*}
w_n(\xi _{w,\eta })-w_n ^{*}(\xi _{w,\eta })=\frac{2}{2+\eta } 
\end{equation*}
for all $2\leq n\leq d$.
\end{thm}

It seems that for each $d\geq 3$, the real numbers $\xi $ defined by Theorems \ref{thm:rmainconti1} and \ref{thm:rmainconti2} are the first explicit continued fraction examples for which $w_d(\xi )$ and $w_d ^{*}(\xi )$ are different.

We denote by $\abs{\cdot }_p$ the absolute value of $\lQ_p$ normalized to satisfy $\abs{p}_p=p^{-1}$.
We recall the definitions of Diophantine exponents $w_n$ and $w_n^{*}$ in $\lQ _p$.
For $\xi \in \lQ_p$ and an integer $n\geq 1$, we denote by $w_n(\xi )$ (resp.\ $w_n^{*}(\xi )$) the supremum of the real numbers $w$ (resp.\ $w^{*}$) which satisfy 
\begin{equation*}
0<|P(\xi )|_p\leq H(P)^{-w-1}\quad (\text{resp.\ } 0<|\xi -\alpha |_p\leq H(\alpha )^{-w^{*}-1})
\end{equation*}
for infinitely many $P(X)\in \lZ[X]$ of degree at most $n$ (resp.\ algebraic numbers $\alpha \in \lQ_p$ of degree at most $n$).

\begin{thm}\label{thm:pmainconti1}
Let $d\geq 2$ be an integer and $w\geq (3d+2+\sqrt{9d^2+4d+4})/2$ be a real number.
Let $b$ be a positive integer and $(\varepsilon _j)_{j\geq 0}$ be a sequence in $\{ 0,1\}$.
Let $(a_{n,w})_{n\geq 1}$ be a sequence given by
\begin{equation*}
a_{n,w}=\begin{cases}
b+3 i+2 & \text{if } n=\lfloor w^i \rfloor \text{ for some integer } i\geq 0,\\
b+3 i+\varepsilon _j & \text{if } \lfloor w^i \rfloor <n<\lfloor w^{i+1} \rfloor \text{ for some integer } i\geq 0.
\end{cases}
\end{equation*}
Set the Schneider's $p$-adic continued fraction $\xi _w:=[a_{1,w},a_{2,w},\ldots ]\in \lQ _p$.
Then we have 
\begin{equation}\label{eq:pmainconti1}
w_n ^{*}(\xi _w)=w-1,\quad w_n(\xi _w)=w 
\end{equation}
for all $2\leq n\leq d$.
\end{thm}

\begin{thm}\label{thm:pmainconti2}
Let $d\geq 2$ be an integer and $w\geq 121d^2$ be a real number.
Let $b$ be positive integer, $(\varepsilon _j)_{j\geq 0}$ be a sequence in $\{ 0,1\}$, and $0<\eta <\sqrt{w}/d$ be a positive number.
Let $(a_{n,w,\eta })_{n\geq 1}$ be the sequence given by
\begin{equation*}
a_{n,w,\eta }=\begin{cases}
b+4 i+3 & \text{if } n=\lfloor w^i \rfloor \text{ for some integer } i\geq 0,\\
b+4 i+2 & \parbox{210pt}{$\text{if } \lfloor w^i \rfloor <n<\lfloor w^{i+1} \rfloor \text{ for some integer } i\geq 0 \text{ and } (n-\lfloor w^i \rfloor )/\lfloor \eta w^i \rfloor \in \lZ ,$}\\
b+4 i+\varepsilon _i & \parbox{210pt}{$\text{if } \lfloor w^i \rfloor <n<\lfloor w^{i+1} \rfloor \text{ for some integer } i\geq 0 \text{ and } (n-\lfloor w^i \rfloor )/\lfloor \eta w^i \rfloor \not\in \lZ .$}
\end{cases}
\end{equation*}
Set the Schneider's $p$-adic continued fraction $\xi _{w,\eta }:=[a_{1,w,\eta },a_{2,w,\eta },\ldots ]\in \lQ _p$.
Then we have
\begin{equation}\label{eq:pmainconti2}
w_n ^{*}(\xi _{w,\eta })=\frac{2 w-2-\eta }{2+\eta },\quad w_n(\xi _{w,\eta })=\frac{2 w-\eta }{2+\eta }
\end{equation}
for all $2\leq n\leq d$.
Hence, we have
\begin{equation*}
w_n(\xi _{w,\eta })-w_n ^{*}(\xi _{w,\eta })=\frac{2}{2+\eta } 
\end{equation*}
for all $2\leq n\leq d$.
\end{thm}

The definition and notation of the Schneider's $p$-adic continued fractions can be found in \cite{Bugeaud15}.
It seems that for each $d\geq 3$, the $p$-adic numbers $\xi $ defined by Theorems \ref{thm:pmainconti1} and \ref{thm:pmainconti2} are the first explicit continued fraction examples for which $w_d(\xi )$ and $w_d ^{*}(\xi )$ are different.

In what follows, we prepare lemmas in order to prove the above theorems.
We omit the details of proofs of these lemmas.

We denote by $\lZ[X]_{\min }$ the set of all non-constant, irreducible, primitive polynomials in $\lZ[X]$ whose leading coefficients are positive.
For $\xi \in \lR$ and an integer $n\geq 1$, we denote by $\tilde{w}_n(\xi )$ the supremum of the real numbers $w$ which satisfy
\begin{equation*}
0<|P(\xi )|\leq H(P)^{-w}
\end{equation*}
for infinitely many $P(X) \in \lZ[X]_{\min }$ of degree at most $n$.
For $\xi \in \lQ_p$ and an integer $n\geq 1$, we denote by $\tilde{w}_n(\xi )$ the supremum of the real numbers $w$ which satisfy
\begin{equation*}
0<|P(\xi )|_p\leq H(P)^{-w-1}
\end{equation*}
for infinitely many $P(X) \in \lZ[X]_{\min }$ of degree at most $n$.

Using Gelfond's Lemma (see Lemma A.3 in \cite{Bugeaud04}) instead of \eqref{eq:heighttwo}, we obtain an analogue of Lemma \ref{lem:weak} for real numbers.

\begin{lem}\label{lem:weakr}
	Let $n\geq 1$ be an integer and $\xi $ be a real number.
	Then we have
	\begin{equation*}
	w_n (\xi )=\tilde{w}_n (\xi ).
	\end{equation*}
\end{lem}

Using Gelfond's Lemma and the fact that $|a||a|_p\geq 1$ for non-zero integers $a$, we obtain an analogue of Lemma \ref{lem:weak} for $p$-adic numbers.

\begin{lem}\label{lem:weakp}
	Let $n\geq 1$ be an integer and $\xi $ be in $\lQ _p$.
	Then we have
	\begin{equation*}
	w_n (\xi )=\tilde{w}_n (\xi ).
	\end{equation*}
\end{lem}

Analogues of Lemma \ref{lem:poly} for real and $p$-adic numbers follow in a similar way to the proof of Lemma \ref{lem:poly}.

\begin{lem}\label{lem:polyr}
	Let $\alpha ,\beta $ be real numbers and $P(X) \in \lZ[X]$ be a non-constant polynomial of degree $d$.
	Let $C\geq 0$ be a real number.
	Assume that $|\alpha -\beta |\leq C$.
	Then there exists a positive constant $C_2(\alpha ,C,d)$, depending only on $\alpha ,C,$ and $d$ such that
	\begin{equation*}
	|P(\alpha )-P(\beta )|\leq C_2(\alpha ,C,d)|\alpha -\beta |H(P).
	\end{equation*}
\end{lem}

\begin{lem}\label{lem:polyp}
	Let $\alpha ,\beta $ be in $\lQ _p$ and $P(X) \in \lZ[X]$ be a non-constant polynomial of degree $d$.
	Let $C\geq 0$ be a real number.
	Assume that $|\alpha -\beta |_p\leq C$.
	Then there exists a positive constant $C_3(\alpha ,C,d)$, depending only on $\alpha ,C,$ and $d$ such that
	\begin{equation*}
	|P(\alpha )-P(\beta )|_p\leq C_3(\alpha ,C,d)|\alpha -\beta |_p.
	\end{equation*}
\end{lem}

Lemmas \ref{lem:weak}, \ref{lem:poly}, \ref{lem:height} \eqref{eq:heighteq}, and  \ref{lem:Galois}--\ref{lem:Lio.inequ2} are used in the proof of Lemma \ref{lem:bestquad}.
Analogues of Lemmas \ref{lem:Galois}--\ref{lem:Lio.inequ2} for real numbers and Lemmas \ref{lem:Galois} and \ref{lem:Lio.inequ2} for $p$-adic numbers are already known (see p.730 in \cite{Bugeaud12}, Theorem A.1 and Corollary A.2 in \cite{Bugeaud04}, and Lemmas 3.2 and 2.5 in \cite{Pejkovic12}).
Slightly modifying the proof of Lemma 2.5 in \cite{Pejkovic12}, we obtain an analogue of Lemma \ref{lem:Lio.inequ1} for $p$-adic numbers.
That is, for non-constant polynomial $P(X)\in \lZ [X]$ of degree $m$ and algebraic number $\alpha \in \lQ_p$ of degree $n$,
\begin{equation}\label{eq:liop}
P(\alpha )=0 \text{ or } |P(\alpha )|_p\geq cH(P)^{-n}H(\alpha )^{-m},
\end{equation}
where $c$ is a positive constant depending only on $m$ and $n$.
The inequality \eqref{eq:liop} is probably known.
However, we were unable to find it in the literature.

Using Lemma A.2 in \cite{Bugeaud04} instead of \eqref{eq:heighteq}, in addition to analogues of \ref{lem:weak}, \ref{lem:poly}, and  \ref{lem:Galois}--\ref{lem:Lio.inequ2}, we obtain an analogue of Lemma \ref{lem:bestquad} for real numbers.

\begin{lem}\label{lem:bestquadr}
	Let $d\geq 2$ be an integer.
	Let $\xi $ be a real number, and $\theta ,\rho ,\delta $ be positive numbers.
	Assume that there exists a sequence $(\alpha _j)_{j\geq 1}$ such that for any $j\geq 1$, $\alpha _j\in \lR$ is quadratic, $(H(\alpha _j))_{j\geq 1}$ is a divergent increasing sequence, and
	\begin{gather*}
	\limsup_{j\rightarrow \infty } \frac{\log H(\alpha _{j+1})}{\log H(\alpha _j)}\leq \theta ,\\
	d+\delta \leq \liminf_{j\rightarrow \infty } \frac{-\log |\xi -\alpha _j|}{\log H(\alpha _j)},\quad \limsup_{j\rightarrow \infty } \frac{-\log |\xi -\alpha _j|}{\log H(\alpha _j)}\leq d+\rho .
	\end{gather*}
	If $2d\theta \leq (d-2+\rho )\delta $, then we have for all $2\leq n\leq d$,
	\begin{equation*}
	d-1+\delta \leq w_n ^{*}(\xi )\leq d-1+\rho .
	\end{equation*}
	Furthermore, assume that there exist a non-negative number $\varepsilon $ and a positive number $c$ such that for any $j\geq 1, 0<|\alpha _j-\alpha _j '|\leq c$ and
	\begin{equation*}
	\limsup_{j\rightarrow \infty }\frac{-\log |\alpha _j -\alpha _j '|}{\log H(\alpha _j)}\geq \varepsilon .
	\end{equation*}
	If $2d\theta \leq (d-2+\delta )\delta $, then we have for all $2\leq n\leq d$,
	\begin{equation*}
	d-1+\delta \leq w_n ^{*}(\xi )\leq d-1+\rho ,\quad \varepsilon \leq w_n(\xi )-w_n ^{*}(\xi ).
	\end{equation*}
	Finally, assume that there exists a non-negative number $\chi $ such that
	\begin{equation*}
	\limsup_{i\rightarrow \infty} \frac{-\log |\alpha _i-\alpha _i '|}{\log H(\alpha _i)}\leq \chi .
	\end{equation*}
	Then we have for all $2\leq n\leq d$,
	\begin{equation*}
	d-1+\delta \leq w_n ^{*}(\xi )\leq d-1+\rho ,\quad \varepsilon \leq w_n(\xi )-w_n ^{*}(\xi )\leq \chi .
	\end{equation*}
\end{lem}

\begin{lem}\label{lem:bestquadp}
	Let $d\geq 2$ be an integer.
	Let $\xi $ be in $\lQ _p$ with $|\xi |_p\leq 1$ and $\theta ,\rho ,\delta $ be positive numbers.
	Assume that there exists a sequence $(\alpha _j)_{j\geq 1}$ such that for any $j\geq 1$, $\alpha _j\in \lQ_p$ is quadratic, $(H(\alpha _j))_{j\geq 1}$ is a divergent increasing sequence, and
	\begin{gather*}
	\limsup_{j\rightarrow \infty } \frac{\log H(\alpha _{j+1})}{\log H(\alpha _j)}\leq \theta ,\\
	d+\delta \leq \liminf_{j\rightarrow \infty } \frac{-\log |\xi -\alpha _j|_p}{\log H(\alpha _j)},\quad \limsup_{j\rightarrow \infty } \frac{-\log |\xi -\alpha _j|_p}{\log H(\alpha _j)}\leq d+\rho .
	\end{gather*}
	If $2d\theta \leq (d-2+\rho )\delta $, then we have for all $2\leq n\leq d$,
	\begin{equation*}
	d-1+\delta \leq w_n ^{*}(\xi )\leq d-1+\rho .
	\end{equation*}
	Furthermore, assume that there exists a non-negative number $\varepsilon $ such that for any $j\geq 1, 0<|\alpha _j-\alpha _j '|_p\leq 1$ and
	\begin{equation*}
	\limsup_{j\rightarrow \infty }\frac{-\log |\alpha _j -\alpha _j '|_p}{\log H(\alpha _j)}\geq \varepsilon .
	\end{equation*}
	If $2d\theta \leq (d-2+\delta )\delta $, then we have for all $2\leq n\leq d$,
	\begin{equation*}
	d-1+\delta \leq w_n ^{*}(\xi )\leq d-1+\rho ,\quad \varepsilon \leq w_n(\xi )-w_n ^{*}(\xi ).
	\end{equation*}
	Finally, assume that there exists a non-negative number $\chi $ such that
	\begin{equation*}
	\limsup_{i\rightarrow \infty} \frac{-\log |\alpha _i-\alpha _i '|_p}{\log H(\alpha _i)}\leq \chi .
	\end{equation*}
	Then we have for all $2\leq n\leq d$,
	\begin{equation*}
	d-1+\delta \leq w_n ^{*}(\xi )\leq d-1+\rho ,\quad \varepsilon \leq w_n(\xi )-w_n ^{*}(\xi )\leq \chi .
	\end{equation*}
\end{lem}

Note that the proof of Lemma \ref{lem:bestquadp} uses analogues of Lemmas \ref{lem:weak}, \ref{lem:poly}, and  \ref{lem:Galois}--\ref{lem:Lio.inequ2} and the following:
Let $P_j(X)=A_j(X-\alpha _j)(X-\alpha _j')\in \lZ[X]_{\min }$ be the minimal polynomial of $\alpha _j$.
Then, by the proof of Lemma 5 in \cite{Bugeaud15}, we have $|A_j|_p=1$ for sufficiently large $j$.

\begin{proof}[Proof of Theorem \ref{thm:rmainconti1}]
	For any $j\geq 2$, we put
	\begin{equation*}
	\xi _{w,j}:=[0,a_{1,w},\ldots ,a_{\lfloor w^j\rfloor ,w},\overline{a}]\in \lR.
	\end{equation*}
	It follows from the proof of Theorem 4.1 in \cite{Bugeaud11} that $\xi _{w,j}$ are quadratic irrationals, $(H(\xi _{w,j}))_{j\geq 1}$ is a divergent increasing sequence, and
	\begin{gather*}
	\lim _{j\rightarrow \infty }\frac{-\log |\xi _w -\xi _{w,j}|}{\log H(\xi _{w,j})}=w,\quad \lim _{j\rightarrow \infty }\frac{-\log |\xi _{w,j} -\xi _{w,j} '|}{\log H(\xi _{w,j})}=1,\\
	\limsup _{j\rightarrow \infty }\frac{\log H(\xi _{w,j+1})}{\log H(\xi _{w,j})}\leq w.
	\end{gather*}
	Hence, we have \eqref{eq:rmainconti1} by Lemma \ref{lem:bestquadr}.
\end{proof}

\begin{proof}[Proof of Theorem \ref{thm:rmainconti2}]
	For any $j\geq 2$, we put
	\begin{equation*}
	\xi _{w,\eta, j}:=[0,a_{1,w,\eta },\ldots ,a_{\lfloor w^j\rfloor ,w,\eta },\overline{a,\ldots ,a,c}]\in \lR,
	\end{equation*}
	where the length of the periodic part is $\lfloor \eta w^j\rfloor $.
	It follows from the proof of Theorem 4.3 in \cite{Bugeaud11} that $\xi _{w,\eta ,j}$ are quadratic irrationals, $(H(\xi _{w,\eta ,j}))_{j\geq 1}$ is a divergent increasing sequence, and
	\begin{gather*}
	\lim _{j\rightarrow \infty }\frac{-\log |\xi _{w,\eta } -\xi _{w,\eta ,j}|}{\log H(\xi _{w,\eta ,j})}=\frac{2 w}{2+\eta },\quad \lim _{j\rightarrow \infty }\frac{-\log |\xi _{w,\eta ,j} -\xi _{w,\eta ,j} '|}{\log H(\xi _{w,\eta ,j})}=\frac{2}{2+\eta },\\
	\limsup _{j\rightarrow \infty }\frac{\log H(\xi _{w,\eta ,j+1})}{\log H(\xi _{w,\eta ,j})}\leq w.
	\end{gather*}
	Hence, we have \eqref{eq:rmainconti2} by Lemma \ref{lem:bestquadr}.
\end{proof}

\begin{proof}[Proof of Theorem \ref{thm:pmainconti1}]
	For any $j\geq 2$, we put
	\begin{equation*}
	\xi _{w,j}:=[a_{1,w},\ldots ,a_{\lfloor w^j\rfloor ,w},\overline{a_{\lfloor w^j\rfloor +1,w}}]\in \lQ_p.
	\end{equation*}
	It follows from the proof of Theorem 1 in \cite{Bugeaud15} that $\xi _{w,j}$ are quadratic irrationals, $(H(\xi _{w,j}))_{j\geq 1}$ is a divergent increasing sequence, and
	\begin{gather*}
	\lim _{j\rightarrow \infty }\frac{-\log |\xi _w -\xi _{w,j}|_p}{\log H(\xi _{w,j})}=w,\quad \lim _{j\rightarrow \infty }\frac{-\log |\xi _{w,j} -\xi _{w,j} '|_p}{\log H(\xi _{w,j})}=1,\\
	\limsup _{j\rightarrow \infty }\frac{\log H(\xi _{w,j+1})}{\log H(\xi _{w,j})}\leq w.
	\end{gather*}
	Hence, we have \eqref{eq:pmainconti1} by Lemma \ref{lem:bestquadp}.
\end{proof}

\begin{proof}[Proof of Theorem \ref{thm:pmainconti2}]
	For any $j\geq 2$, we put
	\begin{equation*}
	\xi _{w,\eta, j}:=[a_{1,w,\eta },\ldots ,a_{\lfloor w^j\rfloor ,w,\eta },\overline{a_{\lfloor w^j\rfloor +1,w,\eta },\ldots ,a_{\lfloor w^j\rfloor +\lfloor \eta w^j\rfloor,w,\eta }}]\in \lQ _p.
	\end{equation*}
	It follows from the proof of Theorem 2 in \cite{Bugeaud15} that $\xi _{w,\eta ,j}$ are quadratic irrationals, $(H(\xi _{w,\eta ,j}))_{j\geq 1}$ is a divergent increasing sequence, and
	\begin{gather*}
	\lim _{j\rightarrow \infty }\frac{-\log |\xi _{w,\eta } -\xi _{w,\eta ,j}|_p}{\log H(\xi _{w,\eta ,j})}=\frac{2 w}{2+\eta },\quad \lim _{j\rightarrow \infty }\frac{-\log |\xi _{w,\eta ,j} -\xi _{w,\eta ,j} '|_p}{\log H(\xi _{w,\eta ,j})}=\frac{2}{2+\eta },\\
	\limsup _{j\rightarrow \infty }\frac{\log H(\xi _{w,\eta ,j+1})}{\log H(\xi _{w,\eta ,j})}\leq w.
	\end{gather*}
	Hence, we have \eqref{eq:pmainconti2} by Lemma \ref{lem:bestquadp}.
\end{proof}

\subsection*{Acknowledgements}
The author would like to express his gratitude to Prof.~Shigeki Akiyama for the helpful comments on Lemma \ref{lem:seq} and improving the language of this paper.
The author is deeply grateful to Prof.~Hajime Kaneko for the helpful comments on Corollary \ref{cor:linind} and Lemma \ref{lem:seq}.
The author also wishes to express his thanks to Prof.~Dinesh S.~Thakur for several helpful comments.
The author would like to thank the referee for careful reading of this paper and giving helpful comments.

\end{document}